\documentclass{proc-l}

\usepackage{amsmath,amssymb,amsthm,amsfonts}
\usepackage{amscd}
\usepackage{tikz}
\usepackage{graphicx}
\usepackage{float}
\usepackage{caption}
\usepackage{subcaption}
\usepackage[english]{babel}
\usepackage[margin=1.5in]{geometry}
\usepackage{pgfpages}

\newtheorem{theorem}{Theorem}[section]
\newtheorem{proposition}[theorem]{Proposition}
\newtheorem{corollary}[theorem]{Corollary}
\newtheorem{lemma}[theorem]{Lemma}
\newtheorem{question}[theorem]{Question}

\theoremstyle{definition}

\newcommand\R{\mathbb R}

\newcommand\N{\mathbb N}
\newcommand\Q{\mathbb Q}

\newcommand\C{\mathbb C}
\newcommand{\abs}[1]{\lvert #1\rvert}

\DeclareMathOperator{\tint}{int}
\begin{document}
\title{Regularity points and Jensen measures for $R(X)$}

\author{J. F. Feinstein}
\address{School of Mathematical Sciences, The University of Nottingham, University Park, Nottingham, NG7 2RD, UK}
\email{joel.feinstein@nottingham.ac.uk}

\author{H. Yang}
\address{School of Mathematical Sciences, The University of Nottingham, University Park, Nottingham, NG7 2RD, UK}
\email{pmxhy1@nottingham.ac.uk}
\thanks{The second author is supported by a China Tuition Fee Research Scholarship and the School of Mathematical Sciences at the University of Nottingham}

\maketitle



\begin{abstract}
We discuss two types of \lq regularity point', points of continuity and R-points for Banach function algebras, which were introduced by the first author and Somerset in \cite{FeinsteinSomerset2000}.  We show that, even for the natural uniform algebras $R(X)$ (for compact plane sets X), these two types of regularity point can be different. We then give a new method for constructing Swiss cheese sets $X$ such that $R(X)$ is not regular, but such that $R(X)$ has no non-trivial Jensen measures. The original construction appears in the first author's previous work. Our new approach to constructing such sets is more general, and allows us to obtain additional properties. In particular, we use our construction to give an example of such a Swiss cheese set $X$ with the property that the set of points of discontinuity for $R(X)$ has positive area.
\end{abstract}

\section{Introduction}

In \cite{FeinsteinSomerset2000}, the first  author and Somerset studied the failure of regularity for Banach function algebras in terms of two types of \lq regularity point': points of continuity and R-points. (These technical terms and others are defined in the next section.) They gave examples of Banach function algebras where points of continuity do not coincide with R-points, but these algebras are not natural. There do not appear to be any examples in the literature of natural Banach function algebras where the two types of regularity point are different. In Section 3 we give examples to show that they can differ even for natural uniform algebras such as $R(X)$, for suitable compact plane sets $X$. These examples are inspired by an example in \cite{2012FeinsteinMortini}, where the first author and Mortini also studied these types of regularity points. (In that paper, these points were also described as regularity points of types I and II). We show that it is possible for $R(X)$ to have exactly one point of continuity while having no R-points, and it is also possible for $R(X)$ to have exactly one R-point while having no points of continuity.

In Section 4 of this note we give a new way to construct Swiss cheese sets $X$ such that $R(X)$ has no non-trivial Jensen measures, but $R(X)$ is not regular. Such examples were first constructed in \cite{2001Feinstein} by the first author. Our new construction is more general, and allows us to obtain additional properties. In particular, we show that in such examples, the ``exceptional set'' of points where regularity fails (the set of points of discontinuity, and also the set of non-R-points) can have positive area. We also provide a more elementary argument to show that there are no non-trivial Jensen measures for our Swiss cheese sets $X$, as an alternative to the method used in \cite{2001Feinstein}, which was based on the deeper theory of the fine topology. In the next section  we describe the background in more detail, and introduce the definitions, terminology and preliminary results we need.

\section{Preliminaries}

Throughout this note, we shall only consider unital, complex algebras. By a compact plane set we mean a \emph{non-empty}, compact subset of the complex plane. For $F\subseteq \C$, we denote the interior of $F$ by $\tint F$.

Let $X$ be a non-empty, compact Hausdorff space.
We denote by $C(X)$ the algebra of all complex-valued continuous functions on $X$. When endowed with the uniform norm, $C(X)$ is a Banach algebra. A \emph{Banach function algebra on $X$}  is a  subalgebra $A$ of $C(X)$ such that $A$ separates the points of $X$ and contains the constant functions, and such that $A$ has a complete algebra norm. We say that $A$ is a \emph{uniform algebra on $X$} if $A$ is a Banach function algebra on $X$ whose norm is the uniform norm on $X$.
Let $A$ be a Banach function algebra on $X$. We often identify points of $X$ with the corresponding evaluation characters on $A$. We say that $A$ is \emph{natural} (on $X$) if these are the only characters on $A$.

Let $A$ be a natural Banach function algebra on $X$, and let $x\in X$. We denote by $J_x$ the ideal of functions $f$ in $A$ such that $x$ is in the interior of the zero set of $f$, $f^{-1}(\{0\})$. We denote by $M_x$ the ideal of functions $f$ in $A$ such that $f(x) = 0$. Recall that, for an ideal $I$ in $A$, the \emph{hull} of $I$, denoted by $h(I)$, is defined by
\[ h(I)=\bigcap_{f\in I}\{z\in X:f(z)=0\}.\]
We say that $x$ is a \emph{point of continuity} (for $A$) if, for all $y \in X \setminus \{x\}$ we have $ J_y\nsubseteq M_x$; 
we say that $x$ is an \emph{R-point} (for $A$) if, for all $y \in X \setminus \{x\}$ we have $J_x\nsubseteq M_y$. 
We say that $A$ is \emph{regular} if, for every closed subset $F$ of $X$ and every $y\in X\setminus F$, there exists $f\in A$ with $f(y)=1$ and $f(F)\subseteq \{0\}$.

It is standard that the natural Banach function algebra $A$ is regular if and only if every point of $X$ is a point of continuity, and this is also equivalent to the condition that every point of $X$ is an R-point (\cite{2012FeinsteinMortini,FeinsteinSomerset2000}).

Let $x \in X$. Then $x$ is a point of continuity if and only if, for all $y\in X\setminus \{x\}$, we have $x\notin h(J_y)$; $x$ is an R-point if and only if $h(J_x)=\{x\}$.

Let $A$ be a uniform algebra on $X$, and let $\varphi$ be a character on $A$.
A \emph{representing measure} $\mu$ for $\varphi$ is a regular, Borel probability measure supported on $X$ such that
\[ \varphi(f) = \int_X f\,{\rm d}\mu \qquad (f\in A).
\]
A \emph{Jensen measure} $\mu$ for $\varphi$ is a representing measure for $\varphi$ such that
\[ \log\abs{\varphi(f)} \leq \int_X \log\abs{f(w)}\,{\rm d}\mu(w) \qquad (f\in A),
\] where $\log(0):=-\infty$.
It is standard (see, for example, \cite[p. 114]{browder1969} or \cite[p. 33]{gamelin1984}) that every character on $A$ has at least one Jensen measure.

Let $x \in X$. We say that a representing measure for $x$ is \emph{trivial} if it is the point mass measure at $x$. As noted in \cite{2001Feinstein}, if $x$ is a point of continuity, then there are no non-trivial Jensen measures for $x$.
In fact the elementary argument that proves this actually establishes the following lemma. (The first part of this result is essentially stated on \cite[page 33]{gamelin1984}.)

\begin{lemma}
\label{Jensen_measures}
Let $A$ be a natural uniform algebra on $X$, and let $x\in X$. Suppose that $x$ has a non-trivial Jensen measure $\mu$, and let $F$ be the closed support of $\mu$. Then, for all $y \in F\setminus \{x\}$, we have $J_y \subseteq M_x$. Thus $x$ is not a point of continuity for $A$, and no points of $F\setminus \{x\}$ can be R-points for $A$.
\end{lemma}

For more details concerning the above definitions, see standard texts on uniform algebras \cite{browder1969,gamelin1984, ELStout1971}, and for  commutative Banach algebras, see \cite{HGDales2000}.

A \emph{Swiss cheese set} is a compact plane set obtained by deleting a sequence of open discs from a closed disc. Since they were first introduced by Roth \cite{roth1938}, there have been numerous applications of Swiss cheese sets in the literature. See, for example, \cite{browder1969,2014Feinstein,gamelin1984,ELStout1971}.

Let $X$ be a compact plane set. By $P(X)$ we denote the set of those functions $f\in C(X)$ which can be uniformly approximated on $X$ by polynomial functions; by $R(X)$ we denote the set of those functions $f\in C(X)$ which can be uniformly approximated on $X$ by rational functions with no poles on $X$; by $A(X)$ we denote the set of those functions $f\in C(X)$ which are holomorphic on the interior of $X$. It is standard that when endowed with the uniform norm, all of $P(X)$, $R(X)$ and $A(X)$ are natural uniform algebra on $X$. It is clear that $P(X)\subseteq R(X)\subseteq A(X)$, and Mergelyan's theorem \cite[p.48]{gamelin1984} asserts that if $\C\setminus X$ is connected, then $P(X)=R(X)=A(X)$.  Note that if $R(X)$ is regular, there are no non-trivial Jensen measures supported on $X$. In \cite{2001Feinstein}, the first author constructed Swiss cheese sets $X$ where $R(X)$ admits no non-trivial Jensen measures while $R(X)$ is not regular. In the examples, the line segment $I=[-1/2,1/2]$ is contained in $X$, and only the points in $X\setminus I$ are points of continuity. The proof in \cite{2001Feinstein} that $R(X)$ has no non-trivial Jensen measures uses the theory of Jensen interior, see \cite[p. 319]{gamelin1983} or papers of Debiard and Gaveau. In Section~\ref{R_points_and_points_of_continuity}, we use a more elementary approach based on regularity points.

\section{R-points and points of continuity for $R(X)$} \label{R_points_and_points_of_continuity}

In this section we construct a compact plane set $X$ such that $R(X)$ has exactly one R-point but has no points of continuity. We also construct a compact plane set $X$ such that $R(X)$ has exactly one point of continuity but has no R-points.

Before we proceed to examples where R-points differ from points of continuity, we give the following lemma, which establishes a relationship between these two types of regularity points.

\begin{lemma} \label{Isolated_Point_Argument}
Let $X$ be a compact Hausdorff space, let $A$ be a natural Banach function algebra on $X$, and let $x\in X$. Suppose that there is an open neighbourhood $U$ of $x$ such that each point in $U\setminus \{x\}$ is a point of continuity for $A$. Then $x$ is an R-point for $A$.
\end{lemma}

\begin{proof}
Since each $y\in U\setminus\{x\}$ is a point of continuity, we have $J_x\nsubseteq M_y$ and thus $y\notin h(J_x)$. This shows $h(J_x)\cap (U\setminus\{x\})=\emptyset$. Since $h(J_x)$ is connected (\cite[Theorem~3.2]{FeinsteinSomerset2000}), we conclude that $h(J_x)=\{x\}$ and $x$ is an R-point.
\end{proof}
In particular, the above lemma holds for $R(X)$.

\medskip

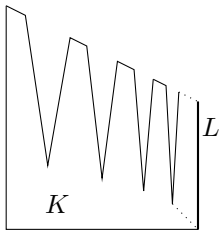
\begin{figure}[h]
\centering
\begin{tikzpicture}[scale=1.7]
\draw (0,1.75)--(0,0)--(1.5,0)--(1.5,1);
\draw [thick] (1.5,0)--(1.5,1);
\draw (0,1.75)--(0.15,1.675)--(0.325,0.5)--(0.5,1.5)--(0.63,1.435)--(0.75,0.4)--(0.87,1.315)--(1,1.25)--(1.075,0.3)--(1.15,1.175);
\draw  (1.15,1.175)--(1.25,1.125)--(1.3,0.2)--(1.35,1.075);
\draw [dotted] (1.35,1.075)--(1.5,1);
\draw [dotted] (1.3,0.2)--(1.5,0);
\node at (0.4,0.2) {$K$};
\node at (1.6,0.8) {$L$};
\end{tikzpicture}
\caption{A compact plane set $K$, where the slits accumulate on the right.} \label{Slit_Domains}
\end{figure}

The building blocks of our examples are closures of \emph{slit domains}, such as the compact plane set $K$ shown in figure~\ref{Slit_Domains}. The slits cut in $K$ become thinner and longer, while accumulating on the whole of the line segment forming the right edge of $K$, which we denote by $L$.

The following proposition is a special case of a more general theorem on prime ends introduced in \cite[Section~2.4]{Pommerenke1992}.

\begin{proposition} \label{Prime_Ends}
Let $K$, $L$ be as decribed above (and shown in figure~\ref{Slit_Domains}). Then the Riemann map $\psi$ from $\tint K$ onto the open unit disc admits a continuous extension from $K$ onto the closed unit disc, such that $\psi(z)\neq \psi(w)$ for all $z\in K\setminus L$ and $w\in K\setminus \{z\}$, and $\psi$ is constant on $L$.
\end{proposition}

We remark that if $f\in A(K)$ and $f$ is constant on the line segment forming the \emph{left} edge of $K$, then an easy application of the reflection principle (see, for example, \cite[Chapter~IX,~Theorem~2.1]{Lang1999}) shows that $f$ is constant on $K$.  A similar compact set $K$ was used in \cite[Example~4.2]{2012FeinsteinMortini}.

Let $K_n$ be a sequence of such compact plane sets side by side, touching, tapering, and accumulating on the right  at a point $x_0$, as shown in figure~\ref{Only_Rpoints}. Set
\[ X = \bigcup_{n=1}^\infty K_n\cup \{x_0\}.\]
This gives us our first example.
Note that, for the remainder of this note, unless otherwise specified, the ideals $J_x$ and $M_x$ considered will always be the ideals in the relevant algebra $R(X)$ under consideration.

\begin{figure}[h]
\centering
\begin{tikzpicture}[scale=2.5]
\draw (0,1.75)--(0,0)--(1.5,0)--(1.5,1);
\draw (0,1.75)--(0.15,1.675)--(0.325,0.5)--(0.5,1.5)--(0.63,1.435)--(0.75,0.4)--(0.87,1.315)--(1,1.25)--(1.075,0.3)--(1.15,1.175);
\draw  (1.15,1.175)--(1.25,1.125)--(1.3,0.2)--(1.35,1.075);
\draw [dotted] (1.35,1.075)--(1.5,1);
\draw [dotted] (1.3,0.2)--(1.5,0);

\draw (0*0.571429+1.5,1.75*0.571429)--(0*0.571429+1.5,0*0.571429)--(1.5*0.571429+1.5,0*0.571429)--(1.5*0.571429+1.5,1*0.571429);
\draw (0*0.571429+1.5,1.75*0.571429)--(0.2*0.571429+1.5,1.65*0.571429)--(0.325*0.571429+1.5,0.5*0.571429)--(0.45*0.571429+1.5,1.525*0.571429)--(0.65*0.571429+1.5,1.425*0.571429)--(0.75*0.571429+1.5,0.4*0.571429)--(0.85*0.571429+1.5,1.325*0.571429)--(1*0.571429+1.5,1.25*0.571429)--(1.075*0.571429+1.5,0.3*0.571429)--(1.15*0.571429+1.5,1.175*0.571429);
\draw  (1.15*0.571429+1.5,1.175*0.571429)--(1.25*0.571429+1.5,1.125*0.571429)--(1.3*0.571429+1.5,0.2*0.571429)--(1.35*0.571429+1.5,1.075*0.571429);
\draw [dotted] (1.35*0.571429+1.5,1.075*0.571429)--(1.5*0.571429+1.5,1*0.571429);
\draw [dotted] (1.3*0.571429+1.5,0.2*0.571429)--(1.5*0.571429+1.5,0*0.571429);

\draw (0*0.326531+2.3571435,1.75*0.326531)--(0*0.326531+2.3571435,0*0.326531)--(1.5*0.326531+2.3571435,0*0.326531)--(1.5*0.326531+2.3571435,1*0.326531);
\draw (0*0.326531+2.3571435,1.75*0.326531)--(0.2*0.326531+2.3571435,1.65*0.326531)--(0.325*0.326531+2.3571435,0.5*0.326531)--(0.45*0.326531+2.3571435,1.525*0.326531)--(0.65*0.326531+2.3571435,1.425*0.326531)--(0.75*0.326531+2.3571435,0.4*0.326531)--(0.85*0.326531+2.3571435,1.325*0.326531)--(1*0.326531+2.3571435,1.25*0.326531)--(1.075*0.326531+2.3571435,0.3*0.326531)--(1.15*0.326531+2.3571435,1.175*0.326531);
\draw  (1.15*0.326531+2.3571435,1.175*0.326531)--(1.25*0.326531+2.3571435,1.125*0.326531)--(1.3*0.326531+2.3571435,0.2*0.326531)--(1.35*0.326531+2.3571435,1.075*0.326531);
\draw [dotted] (1.35*0.326531+2.3571435,1.075*0.326531)--(1.5*0.326531+2.3571435,1*0.326531);

\draw (0*0.186589+2.84694,1.75*0.186589)--(0*0.186589+2.84694,0*0.186589)--(1.5*0.186589+2.84694,0*0.186589)--(1.5*0.186589+2.84694,1*0.186589);
\draw (0*0.186589+2.84694,1.75*0.186589)--(0.2*0.186589+2.84694,1.65*0.186589)--(0.325*0.186589+2.84694,0.5*0.186589)--(0.45*0.186589+2.84694,1.525*0.186589)--(0.65*0.186589+2.84694,1.425*0.186589)--(0.75*0.186589+2.84694,0.4*0.186589)--(0.85*0.186589+2.84694,1.325*0.186589)--(1*0.186589+2.84694,1.25*0.186589)--(1.075*0.186589+2.84694,0.3*0.186589)--(1.15*0.186589+2.84694,1.175*0.186589);
\draw  (1.15*0.186589+2.84694,1.175*0.186589)--(1.25*0.186589+2.84694,1.125*0.186589)--(1.3*0.186589+2.84694,0.2*0.186589)--(1.35*0.186589+2.84694,1.075*0.186589);
\draw [dotted] (1.35*0.186589+2.84694,1.075*0.186589)--(1.5*0.186589+2.84694,1*0.186589);

\draw (0,0)--(3.5,0);
\draw [dotted] (1.5*0.186589+2.84694,1*0.186589)--(3.5,0);
\draw [fill] (3.5,0) circle [radius=0.015];
\node at (3.6,-0.05) {$x_0$};

\node at (0.7,-0.2) {$K_1$};
\node at (1.9,-0.2) {$K_2$};
\node at (2.6,-0.2) {$K_3$};
\node at (3.05,-0.2) {$K_4$};

\end{tikzpicture}
\caption{A compact set $X$, where $R(X)$ admits exactly one R-point (the point $x_0$), but no points of continuity.} \label{Only_Rpoints}
\end{figure}
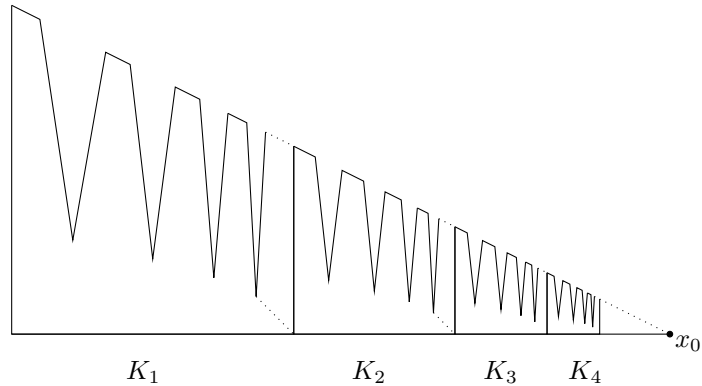

\begin{theorem} \label{Thm_One_Rpoint}
Let $X$ be the compact plane set constructed above. Then $\{x_0\}$ is the only R-point for $R(X)$, while $R(X)$ has no points of continuity.
\end{theorem}

\begin{proof}
For each $n \in \N$, let $L_n$ be the line segment forming the right edge of $K_n$, and let $F_n$ be the trapezoid-shaped compact domain obtained by filling the slits in $K_n$.

Since $\C\setminus X$ is connected, $R(X)=A(X)$ by Mergelyan's theorem. Let $n\in \N$, let $x\in K_n$, and let $f \in J_x$.
Then $f$ vanishes identically on $K_n$, which forces $f$ to vanish identically on $K_m$ for all $m\geq n$, and hence also at $x_0$. This shows that $J_x\subseteq M_y$ for all $y\in \bigcup_{m\geq n}K_m \cup \{x_0\}$. Since $n$ is arbitrary, it follows that no points in $X$ can be points of continuity for $R(X)$, and no points in $X \setminus \{x_0\}$ can be R-points. Again let $x\in K_n$ for some $n\in \N$. We show that $J_{x_0}\nsubseteq M_x$. Set
\[ F = K_{n+1}\cup\left(\bigcup_{\ell=1}^n F_\ell\right).\] Then $x\in F$, and $x$ is not on $L_{n+1}$. Note that $F$ is the closure of a slit domain with right hand segment $L_{n+1}$. By Proposition~\ref{Prime_Ends}, The Riemann map $\psi$ from $\tint F$ to the open unit disc has a continuous extension to $F$. We denote the extended function by $\psi$ again, and notice that $\psi(z)\neq \psi(w)$ for $z\in F\setminus L_{n+1}$, and $w\in F\setminus \{z\}$. We denote the constant value taken by $\psi$ on $L_{n+1}$ by $c$, and set
\[ f(z) = \begin{cases}
\psi(z)-c & \text{if }z\in F, \\
0 & \text{if } z\in X\setminus F.\\ \end{cases}\] Then it is clear that the restriction of $f$ to $X$ is in  $A(X)=R(X)$ and $f\in J_{x_0}\setminus M_x$. Thus $J_{x_0}\nsubseteq M_x$. This shows that $x_0$ is an R-point for $R(X)$. By above, $x_0$ is the only R-point.
\end{proof}

\begin{figure}[h]
\centering
\begin{tikzpicture}[scale=1.7]
\draw (0,1.75)--(0,0)--(1.5,0)--(1.5,1);
\draw [thick] (0,1.75)--(0,0);
\draw [dotted] (0,1.75)--(0.1,1.7);
\draw (0.1,1.7)--(0.135,0.2)--(0.17,1.665)--(0.25,1.625)--(0.3,0.3)--(0.35,1.575)--(0.5,1.5)--(0.625,0.4)--(0.75,1.375);
\draw  (0.75,1.375)--(1,1.25)--(1.175,0.5)--(1.35,1.075);
\draw (1.35,1.075)--(1.5,1);
\draw [dotted] (0.135,0.2)--(0,0);
\node at (0.6,0.2) {$K$};
\node at (-0.1,1.3) {$L$};
\end{tikzpicture}
\caption{A compact plane set $K$, where the slits accumulate on the left.} \label{Slit_Domain_Accumulate_Left}
\end{figure}
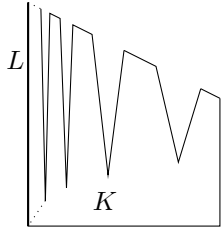

Next we modify the example constructed above to show that there exist compact plane sets $X$, such that $R(X)$ has exactly one point of continuity, but has no R-points. Let $K$ be the closure of a slit domain, with the slits accumulating on the whole of the line segment forming the \emph{left} edge, $L$, of $K$, as shown in  figure~\ref{Slit_Domain_Accumulate_Left}. As in Proposition~\ref{Prime_Ends}, the Riemann mapping $\psi$ from $\tint K$ onto the open unit disc admits a continuous extension from $K$ onto the closed unit disk, such that $\psi(z)\neq \psi(w)$ for all $z\in K\setminus L$ and $w\in K\setminus \{x\}$, and $\psi$ is constant on $L$.
Note that this time, if $f \in A(K)$ and $f$ is constant on the line segment forming the \emph{right} edge of $K$, then $f$ is constant on $K$.

Let $K_n$ be a sequence of such compact plane sets side by side, touching, tapering, and accumulating on the right  at a point $x_0$, as shown in figure~\ref{Only_Point_Of_Continuity}.

\begin{figure}[h]
\centering
\begin{tikzpicture}[scale=2.5]
\draw (0,1.75)--(0,0)--(1.5,0)--(1.5,1);
\draw [dotted] (0,1.75)--(0.1,1.7);
\draw (0.1,1.7)--(0.135,0.2)--(0.17,1.665)--(0.25,1.625)--(0.3,0.3)--(0.35,1.575)--(0.5,1.5)--(0.625,0.4)--(0.75,1.375);
\draw  (0.75,1.375)--(1,1.25)--(1.175,0.5)--(1.35,1.075);
\draw (1.35,1.075)--(1.5,1);
\draw [dotted] (0.135,0.2)--(0,0);

\draw (0*0.571429+1.5,1.75*0.571429)--(0*0.571429+1.5,0*0.571429)--(1.5*0.571429+1.5,0*0.571429)--(1.5*0.571429+1.5,1*0.571429);
\draw [dotted] (0*0.571429+1.5,1.75*0.571429)--(0.1*0.571429+1.5,1.7*0.571429);
\draw (0.1*0.571429+1.5,1.7*0.571429)--(0.135*0.571429+1.5,0.2*0.571429)--(0.17*0.571429+1.5,1.665*0.571429)--(0.25*0.571429+1.5,1.625*0.571429)--(0.3*0.571429+1.5,0.3*0.571429)--(0.35*0.571429+1.5,1.575*0.571429)--(0.5*0.571429+1.5,1.5*0.571429)--(0.625*0.571429+1.5,0.4*0.571429)--(0.75*0.571429+1.5,1.375*0.571429);
\draw  (0.75*0.571429+1.5,1.375*0.571429)--(1*0.571429+1.5,1.25*0.571429)--(1.175*0.571429+1.5,0.5*0.571429)--(1.35*0.571429+1.5,1.075*0.571429);
\draw (1.35*0.571429+1.5,1.075*0.571429)--(1.5*0.571429+1.5,1*0.571429);
\draw [dotted] (0.135*0.571429+1.5,0.2*0.571429)--(0*0.571429+1.5,0*0.571429);

\draw (0*0.326531+2.3571435,1.75*0.326531)--(0*0.326531+2.3571435,0*0.326531)--(1.5*0.326531+2.3571435,0*0.326531)--(1.5*0.326531+2.3571435,1*0.326531);
\draw [dotted] (0*0.326531+2.3571435,1.75*0.326531)--(0.1*0.326531+2.3571435,1.7*0.326531);
\draw (0.1*0.326531+2.3571435,1.7*0.326531)--(0.135*0.326531+2.3571435,0.2*0.326531)--(0.17*0.326531+2.3571435,1.665*0.326531)--(0.25*0.326531+2.3571435,1.625*0.326531)--(0.3*0.326531+2.3571435,0.3*0.326531)--(0.35*0.326531+2.3571435,1.575*0.326531)--(0.5*0.326531+2.3571435,1.5*0.326531)--(0.625*0.326531+2.3571435,0.4*0.326531)--(0.75*0.326531+2.3571435,1.375*0.326531);
\draw  (0.75*0.326531+2.3571435,1.375*0.326531)--(1*0.326531+2.3571435,1.25*0.326531)--(1.175*0.326531+2.3571435,0.5*0.326531)--(1.35*0.326531+2.3571435,1.075*0.326531);
\draw (1.35*0.326531+2.3571435,1.075*0.326531)--(1.5*0.326531+2.3571435,1*0.326531);
\draw [dotted] (0.135*0.326531+2.3571435,0.2*0.326531)--(0*0.326531+2.3571435,0*0.326531);

\draw (0*0.186589+2.84694,1.75*0.186589)--(0*0.186589+2.84694,0*0.186589)--(1.5*0.186589+2.84694,0*0.186589)--(1.5*0.186589+2.84694,1*0.186589);
\draw [dotted] (0*0.186589+2.84694,1.75*0.186589)--(0.1*0.186589+2.84694,1.7*0.186589);
\draw (0.1*0.186589+2.84694,1.7*0.186589)--(0.135*0.186589+2.84694,0.2*0.186589)--(0.17*0.186589+2.84694,1.665*0.186589)--(0.25*0.186589+2.84694,1.625*0.186589)--(0.3*0.186589+2.84694,0.3*0.186589)--(0.35*0.186589+2.84694,1.575*0.186589)--(0.5*0.186589+2.84694,1.5*0.186589)--(0.625*0.186589+2.84694,0.4*0.186589)--(0.75*0.186589+2.84694,1.375*0.186589);
\draw  (0.75*0.186589+2.84694,1.375*0.186589)--(1*0.186589+2.84694,1.25*0.186589)--(1.175*0.186589+2.84694,0.5*0.186589)--(1.35*0.186589+2.84694,1.075*0.186589);
\draw (1.35*0.186589+2.84694,1.075*0.186589)--(1.5*0.186589+2.84694,1*0.186589);

\draw (0,0)--(3.5,0);
\draw [dotted] (1.5*0.186589+2.84694,1*0.186589)--(3.5,0);
\draw [fill] (3.5,0) circle [radius=0.015];
\node at (3.6,-0.05) {$x_0$};

\node at (0.7,-0.2) {$K_1$};
\node at (1.9,-0.2) {$K_2$};
\node at (2.6,-0.2) {$K_3$};
\node at (3.05,-0.2) {$K_4$};

\end{tikzpicture}
\caption{A compact set $X$, where $R(X)$ admits exactly one point of continuity (the point $x_0$), but no R-points.} \label{Only_Point_Of_Continuity}
\end{figure}
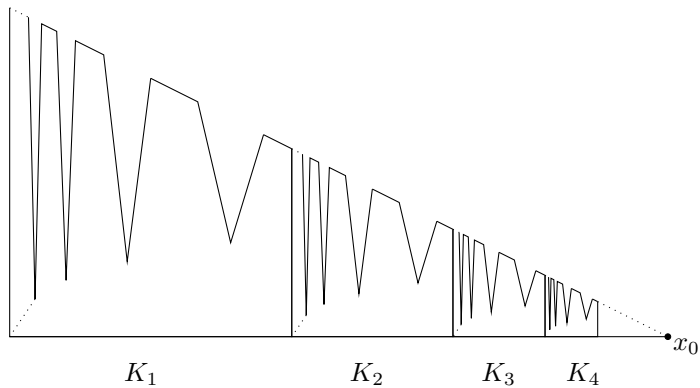

Set
\[ X = \{x_0\} \cup\left(\bigcup_{n=1}^\infty K_n \right).\]
This gives us our second example.

\begin{theorem} \label{Thm_One_Point_Of_Continuity}
Let $X$ be the compact plane set constructed above. Then $\{x_0\}$ is the only point of continuity for $R(X)$, while $R(X)$ has no R-points.
\end{theorem}
\begin{proof}
For each $n\in \N$, let $L_n$ be the line segment forming the left edge of $K_n$, and let $F_n$ be the trapezoid-shaped compact domain obtained by filling the slits in $K_n$.

Since $\C\setminus X$ is connected, $R(X) = A(X)$ by Mergelyan's theorem. Let $x\in X$, and let $f\in J_x$. Then $f$ vanishes on an open neighbourhood of a point in some $K_n$.  This forces $f$ to vanish on all $K_m$ with $m\leq n$, and shows that $J_x\subseteq M_y$ for all $y\in \bigcup_{\ell = 1}^n K_\ell$. Thus $x$ is not an R-point. Let $x\in K_m$ for some $m$, and let $y\in K_n$ for some $n>m$. If $f\in J_y$, then $f$ vanishes identically on all $K_\ell$ for $\ell\leq n$. This forces $f\in M_x$, which shows that $x$ is not a point of continuity. Next we show that $x_0$ is a point of continuity for $R(X)$. Let $y\in X\setminus \{x_0\}$, then $y\in K_n$ for some $n$. Set
\[ F = K_{n+2}\cup \left( \bigcup_{\ell=n+3}^\infty F_\ell \right) \cup \{x_0\},\] then $F$ is the closure of a slit domain, where the slits accumulate on the left edge $L_{n+2}$, but where the right edge of $F$ is degenerate (this does not affect the properties that we need here). The Riemann map $\psi$ from $\tint F$ onto the open unit disc has a continuous extension from $F$ onto the closed unit disc. We denote the extended function by $\psi$ again, and notice that $\psi(z)\neq \psi(w)$ for $z\in F\setminus L_{n+2}$, and $w\in F\setminus \{z\}$. As before, $\psi$ is constant on $L_{n+2}$. We denote this constant value by $c$, and set
\[ f(z) = \begin{cases}
\psi(z)-c & \text{if }z\in F, \\
0 & \text{if } z\in X\setminus F.\\ \end{cases}
\] Then it is clear that the restriction of $f$ to $X$ is in $A(X)=R(X)$ and $f\in J_y\setminus M_{x_0}$. 
This shows that $J_y\nsubseteq M_{x_0}$, and thus $x_0$ is a point of continuity. By above, $x_0$ is the only point of continuity for $R(X)$.
\end{proof}

\section{Trivial Jensen measures without regularity}

In \cite{2001Feinstein}, the first author constructed a Swiss cheese set $X$ such that $R(X)$ has no non-trivial Jensen measures, but $R(X)$ is not regular. In this section we give a new way to construct such sets $X$. Our new approach also allows us to obtain some additional properties.

We begin with a lemma concerning representing measures for $R(X)$.

\begin{lemma} \label{lemma_4_1}
Let $X$ be a compact plane set, let $F$ be a non-empty closed subset of $X$, and let $x \in F$. Suppose that no bounded component of $\C \setminus F$ is contained in $X$, and there exists a non-trivial representing measure $\mu$ for $x$ with respect to $R(X)$, whose closed support is contained in $F$. Then $\mu$ is also a non-trivial representing measure for $x$ with respect to $R(F)$, and $R(F)\neq C(F)$.
\end{lemma}

\begin{proof}
Since  no bounded component of $\C \setminus F$ is contained in $X$, the algebra of functions that are restrictions of functions in $R(X)$ to $F$ is dense in $R(F)$ by Runge's theorem \cite[p. 28]{gamelin1984}.  It follows easily that $\mu$ is also a non-trivial representing measure for $x$ with respect to $R(F)$. Since $C(F)$ has no non-trivial representing measures, we must have $R(F)\neq C(F)$.
\end{proof}

We remark that the compact subset $F$ mentioned in the above lemma must have positive area (by the Hartogs-Rosenthal theorem \cite[p. 161]{browder1969}), and $F$ cannot have empty interior and connected complement in $\C$ (by Mergelyan's theorem, or Lavrentiev's theorem \cite[p. 48]{gamelin1984}). We also remark that in the above lemma, the condition that no bounded component of $\C\setminus F$ be contained in $X$ cannot be omitted. To see this, consider the disc algebra, which is equal to $R(\overline\Delta)$, where $\Delta$ is the open unit disc. Let $F$ be the union of the unit circle and the origin. Then there is a non-trivial representing measure $\mu$ supported on $F$ for the origin with respect to $R(\overline\Delta)$, while $R(F)=C(F)$ by the Hartogs--Rosenthal theorem. Of course this measure $\mu$ is \emph{not} a representing measure with respect to $R(F)$.




\medskip

In our construction below, we will need a lemma of McKissick. The following version of McKissick's lemma was given by K\"orner in \cite{korner1986cheaper}.

\begin{lemma} \label{MK_Lemma}
Let $D$ be an open disc in $\C$ and let $\varepsilon >0$. Then there is a sequence $\Delta_k (k\in \N)$ of pairwise disjoint open discs with each $\Delta_k\subseteq D$ such that the sum of the radii of the $\Delta_k$ is less than $\varepsilon$ and such that, setting $U=\bigcup_{k\in \N}\Delta_k$, there is a sequence $f_n$ of rational functions with poles only in $U$ and such that $f_n$ converges uniformly on $\C\setminus U$ to a function $F$ such that $F(z)=0$ for all $z\in \C\setminus D$ while $F(z)\neq 0$ for all $z\in D\setminus U$.
\end{lemma}

We shall also require the following lemma of Denjoy (\cite{Denjoy1921}) and Carleman (\cite{Carleman1922}). The original lemma is stated for closed intervals in $\R$, but there is no difficulty in replacing a closed interval by a line segment in $\C$. Note that we define derivatives here as limits of quotients using points in the line segment, and all line segments in this note are not degenerate.

\begin{lemma} \label{DC_Lemma}
Let $f$ be an infinitely differentiable function on a line segment $I$ in $\C$ such that
\begin{equation} \label{DC_Condition} \sum_{k=1}^\infty \frac 1{\abs{f^{(k)}}_I^{1/k}}=\infty.\end{equation}
Suppose that there is an $x\in I$ such that $f^{(k)}(x)=0$ for all $k\geq 0$. Then $f$ is constantly $0$ on $I$.
\end{lemma}

We also need the following well-known result concerning derivatives of rational functions. (For an explicit proof of this estimate in the case of the first derivative, see \cite[Theorem~3.2]{1991Feinstein}.)
\begin{lemma} \label{Cauchy_Estimate}
Let $D_n$ be a sequence of open discs in $\C$ (not necessarily pairwise disjoint), and set $X=\overline \Delta\setminus \cup_{n=1}^\infty D_n$. Suppose that $z\in X$. Let $s_n$ denote the distance from $D_n$ to $z$ and $r_n$ the radius of $D_n$. We also set $r_0=1$ and $s_0=1-\abs{z}$. Suppose that $s_n>0$ for all $n$. Then, for all $f\in R_0(X)$ and $k\geq 0$, we have
\[ \abs{f^{(k)}(z)}\leq k!\sum_{j=0}^\infty \frac{r_j}{s_j^{k+1}}\abs{f}_X. \]
\end{lemma}


Let $K$ be a compact subset of $\Delta$. Our new approach allows us to construct Swiss cheese sets $X\supseteq K$ such that $R(X)$ has the kind of properties we want. Usually we will be interested in sets $K$ which contain some line segments.

\begin{theorem} \label{JY_Theorem}
Let $K$ be a compact subset of $\Delta$. Then there is a Swiss cheese set $X$ with $K\subseteq X$ such that every point of $X\setminus K$ is both a point of continuity and an R-point for $R(X)$, but such that for every closed  line segment $I$ contained in $K$, every $f\in R(X)$ is infinitely differentiable on $I$ and satisfies condition \eqref{DC_Condition}.
\end{theorem}

\begin{proof}
Let $\mathcal B$ be the set of all open discs $B$ in the plane such that the centre of $B$ is in $\Q+\Q i$, the radius of $B$ is in
$\Q \cap (0,\infty)$, and $\overline B \subseteq \overline \Delta \setminus K$. Let $(B_n)$ be a sequence enumerating the countable set $\mathcal B$. For each $n\in \N$, let $d_n$ be the distance from $B_n$ to $K$, and choose $\varepsilon_n>0$ small enough so that the inequality
$\varepsilon_n \leq d_n^{k+1}(\log(k +3))^k/2^n$
holds for all $k\in \N$.
Note that, for all $k \in \N$, we have
\begin{equation}
\label{epsilon_bound}
\sum_{n=1}^\infty \frac{\varepsilon_n}{d_n^{k+1}} \leq (\log(k +3))^k\,.
\end{equation}

Now we apply Lemma \ref{MK_Lemma} to each $B_n$, with $\varepsilon$ in the lemma replaced by $\varepsilon_n$. We denote the sequence of open discs obtained from the lemma by $(D_{n,m})_{m\in \N}$, and we denote the radius of $D_{n,m}$ by $r_{n,m}$. Note that $\sum_{m=1}^\infty r_{n,m} <\varepsilon_n$ for each $n\in \N$. Set
\[
X=\overline{\Delta}\setminus \bigcup_{n,m\in \N}D_{n,m}\,.
\]

Let $z\in K$, and let $f \in R_0(X)$. For $n,m \in \N$, let $s_{n,m}$ be the distance from $z$ to $D_{n,m}$. Note that $s_{n,m} \geq d_n$, since $D_{n,m} \subseteq B_n$. Let $d_0$ be the distance from $K$ to the unit circle.
By Lemma~\ref{Cauchy_Estimate}, we have
\begin{eqnarray}
\abs{f^{(k)}(z)} &\leq& k!\left(\frac 1{(1-\abs{z})^{k+1}} + \sum_{n=1}^\infty \sum_{m=1}^\infty \frac{r_{n,m}}{s_{n,m}^{k+1}}\right)\abs{f}_X \nonumber \\
&\leq& k!\left(\frac1{d_0^{k+1}}+\sum_{n=1}^\infty \sum_{m=1}^\infty \frac{r_{n,m}}{d_n^{k+1}}\right)\abs{f}_X \nonumber \\
&\leq& k!\left(\frac1{d_0^{k+1}}+\sum_{n=1}^\infty \frac{\varepsilon_n}{d_n^{k+1}} \right)\abs{f}_X \nonumber \\
&\leq& k!\left(\frac1{d_0^{k+1}}+(\log(k+3))^k\right)\abs{f}_X\,, \label{Derivative_Estimate}
\end{eqnarray}
by (\ref{epsilon_bound}).
Let $I$ be a closed line segment in $K$, and let $f\in R(X)$. Then it follows from inequality \eqref{Derivative_Estimate} that every $f\in R(X)$ is differentiable on $I$, and satisfies
\[\abs{f^{(k)}}_I\leq k!\left( \frac 1{d_0^{k+1}}+((\log(k+3))^k \right)\abs{f}_X. \]
Now we pick $N\in \N$ large enough such that $(\log(k+3))^k \geq 1/d_0^{k+1}$ for all $k\geq N$. Then we have
\begin{eqnarray}
\sum_{k=1}^\infty \frac 1{\abs{f^{(k)}}_I^{1/k}} &\geq& \sum_{k=N}^\infty \frac 1{\abs{f^{(k)}}_I^{1/k}} \nonumber \\
&\geq& \sum_{k=N}^\infty \frac 1{(2\abs{f}_X)^{1/k}(k!)^{1/k}\log(k+3)}\nonumber \\ &\geq& \sum_{k=N}^\infty \frac 1{(2\abs{f}_X)^{1/k}k\log(k+3)}=\infty. \nonumber
\end{eqnarray} Thus $f$ satisfies condition \eqref{DC_Condition} on $I$.

Let $z\in X\setminus K$. We show that $z$ is a point of continuity for $R(X)$. Let $w\in X$ with $w\neq z$. Then we can find some $B_n$ with $z\in B_n$ and $w\in X\setminus \bar{B}_n$. By Lemma~\ref{MK_Lemma}, we can find $f\in R(X)$ with $f(z)\neq 0$ and such that $f$ vanishes on the whole of $X\setminus\bar{B}_n$. Then $f \in J_w \setminus M_z$, and so $J_w \nsubseteq M_z$. Thus $z$ is a point of continuity for $R(X)$.

Since every point of $X \setminus K$ is a point of continuity for $R(X)$, by Lemma~\ref{Isolated_Point_Argument} all of these points are also R-points for $R(X)$.
\end{proof}

Next we give a corollary to show that, for suitable $K$, we can eliminate all the non-trivial Jensen measures for $R(X)$, where $X$ is as constructed in the above theorem.

\begin{corollary}\label{Main_Corollary}
Let $K$ be a compact subset of $\Delta$. Suppose that $\tint K = \emptyset$, $\C \setminus K$ is connected, and $K$ contains at least one closed line segment. Let $X$ be the Swiss cheese set constructed in Theorem \ref{JY_Theorem}, and let $I$ be a line segment contained in $K$. Then $R(X)$ has no non-trivial Jensen measures, but no point of $I$ is an R-point or a point of continuity for $R(X)$, and so $R(X)$ is not regular.
\end{corollary}

\begin{proof}
Let $I$ be a closed line segment in $K$. By Lemma~\ref{DC_Lemma}, if $f\in R(X)$ vanishes on a neighbourhood of a point $x$ in $I$, then $f$ vanishes on $I$. Thus no point of $I$ can be an R-point or a point of continuity.

Certainly there are no non-trivial Jensen measures for points in $X\setminus K$. It remains to show that for each $z\in K$ there are no non-trivial Jensen measures for $z$. Recall that every point of $X \setminus K$ is an R-point for $R(X)$.

Let $z\in K$. Suppose, for contradiction, that $\mu$ is a non-trivial Jensen measure for $z$ with closed support $F$. By Mergelyan's theorem, $R(K)=C(K)$, and so, by Lemma~\ref{lemma_4_1}, $F$ is not contained in $K$. Let $w\in F\setminus K$. Then $w$ is an R-point for $R(X)$. However, this contradicts Lemma~\ref{Jensen_measures}. The result follows.
\end{proof}

We remark that \cite[Corollary~6]{2001Feinstein} is now the special case of Corollary~\ref{Main_Corollary} where we take $K$ to be the closed line segment $[-1/2,1/2]$. We also remark that an alternative approach to the last part of our proof is to note that the Jensen interior of $K$ must be empty, and then use the approach from \cite{2001Feinstein}. Our approach appears to be more elementary, although Lemma~\ref{Isolated_Point_Argument} uses \cite[Theorem~3.2]{FeinsteinSomerset2000}, and that result uses the Shilov Idempotent Theorem. However, for compact plane sets $X$ with empty interior, the special case of \cite[Theorem~3.2]{FeinsteinSomerset2000} for $R(X)$ can be proved without appealing to the Shilov idempotent theorem.

\medskip

Finally we give a concrete example where $K$ has positive area and such that, for the Swiss cheese set $X$ constructed in Theorem~\ref{JY_Theorem}, the set of points of continuity for $R(X)$ is precisely $X\setminus K$, and this is also equal to the set of R-points for $R(X)$.
For this purpose, we use \emph{fat Cantor} sets. A fat Cantor set, also known as a \emph{Smith-Volterra-Cantor} set, is a compact subset of $\R$ that is nowhere dense and has positive length. The Cartesian product of a fat Cantor set and a closed interval can be regarded as a compact subset $K$ of $\C$ with positive area such that $\C \setminus K$ is connected and $K$ has no interior points. Moreover, $K$ is the union of the closed line segments which are contained in $K$.

The following Theorem is now an immediate consequence of Theorem~\ref{JY_Theorem} and Corollary~\ref{Main_Corollary}.

\begin{theorem}
Let $F\subseteq [-1/2,1/2]$ be a fat Cantor set of positive length, and set
\[ K=\{ z\in \C: {\rm Re}(z)\in F, {\rm Im}(z)\in [-1/2,1/2] \}. \]
Then there is a Swiss cheese set $X$ with $K\subseteq X$ such that every point of $X\setminus K$ is a point of continuity and an R-point for $R(X)$, every point in $K$ is a non R-point and a point of discontinuity, and $R(X)$ has no non-trivial Jensen measures.
\end{theorem}

We end this note by asking whether Lemma~\ref{Isolated_Point_Argument} remains true if we interchange \emph{R-point} with \emph{point of continuity}.

\begin{question}
\label{Rquestion}
Let $X$ be a compact Hausdorff space, let $A$ be a natural Banach function algebra on $X$, and let $x\in X$. Suppose that there is an open neighbourhood $U$ of $x$ such that each point in $U\setminus \{x\}$ is an R-point for $A$. Is $x$ necessarily a point of continuity for $A$?
\end{question}

It is an easy consequence of \cite[Proposition~4.1]{FeinsteinSomerset2000} that question \ref{Rquestion} has a positive answer if the algebra $A$ is local (or even 2-local). In particular, the answer is positive for $R(X)$.

\bibliographystyle{amsplain}





\end{document}